\documentclass[11pt,a4paper]{article}
\usepackage[margin=2.5cm]{geometry}
\usepackage{amsmath,amssymb,amsthm}
\usepackage{mathtools}
\usepackage{dsfont}
\usepackage{booktabs}
\usepackage{array}
\usepackage{microtype}
\usepackage{float}
\usepackage{graphicx}
\usepackage[table]{xcolor}
\usepackage{siunitx}
\usepackage{pifont}
\usepackage{algorithm}
\usepackage{algpseudocode}
\usepackage[colorlinks=true,linkcolor=blue!50!black,citecolor=blue!50!black,urlcolor=blue!50!black]{hyperref}
\usepackage{enumitem}

\newcommand{\R}{\mathbb{R}}
\newcommand{\N}{\mathbb{N}}
\newcommand{\inner}[2]{\langle #1,\,#2\rangle}
\newcommand{\SW}{\mathrm{SW}}

\DeclareMathOperator*{\argmax}{arg\,max}

\newcommand{\Fnorm}[1]{\lVert #1\rVert_F}
\newcommand{\lamstar}{\lambda^{\star}}
\newcommand{\Szero}{\mathcal{S}_0}
\newcommand{\Sone}{\mathcal{S}_1}
\newcommand{\Stwo}{\mathcal{S}_2}
\newcommand{\hgl}[1]{\cellcolor{blue!9}#1}

\newtheorem{theorem}{Theorem}[section]
\newtheorem{proposition}[theorem]{Proposition}
\newtheorem{lemma}[theorem]{Lemma}

\theoremstyle{definition}
\newtheorem{definition}[theorem]{Definition}
\theoremstyle{remark}

\title{\bfseries Topological Clustering via Sliced Wasserstein Kernels}
\author{Vikram Aithal$^{1}$, Ajit Kumar$^{1}$, Ambika Sharma$^{2}$\\[4pt]
\normalsize $^{1}$Department of Mathematics, Institute of Chemical Technology, Mumbai\\
\normalsize $^{2}$Department of Mathematics, Bhavan's Hazarimal Somani College, Mumbai\\
\normalsize \texttt{vt.aithal@ictmumbai.edu.in, a.kumar@ictmumbai.edu.in,}\\
\normalsize \texttt{ambikasharma@bhavanschowpatty.ac.in}}
\date{}

\begin{document}
\maketitle

\begin{abstract}
\noindent
Topological data analysis (TDA) uses topological techniques to extract
meaningful shape-based information from complex datasets. Clustering is a
central problem in data analysis, and there has been considerable recent
interest in understanding how TDA can inform it. Existing approaches either
cluster persistence diagrams directly under Wasserstein-type distances, which
is computationally expensive, or use vector representations of diagrams. We propose a kernel $k$-means algorithm built on
a convex combination of sliced Wasserstein (SW) kernels, one for each homology under consideration. Unlike other vector representations of persistence diagrams, the SW kernel is both stable
and discriminative with respect to the $1$-Wasserstein distance. The
method outperforms the baselines on two benchmark datasets and remains competitive on a third synthetic dataset, while being computationally efficient. It also outperforms both a single SW kernel on the union
of all homology groups and an SW kernel
computed directly on the point clouds. The convex
combination assigns an interpretable weight to each $q$-homology kernel. We further validate that the weights identify the discriminating homology.
\medskip

\end{abstract}

\section{Introduction}
\label{sec:intro}

Topological data analysis (TDA) provides a framework for capturing the multi-scale structure of data using tools from algebraic topology with persistent homology as its main tool.
Its origins lie in the size functions of Frosini~\cite{frosini1992measuring},
and the field took its modern form with the introduction of persistent
homology by Edelsbrunner et al.~\cite{edelsbrunner2002topological} and
Zomorodian and Carlsson~\cite{zomorodian2005computing}. Persistent homology
records the topological features of a dataset across all scales and summarises as persistence diagrams: multisets of points in the
plane whose coordinates are the birth and death scales of the
features~\cite{ghrist2008barcodes}. It applies to very general data,
including point clouds, functions, and networks, and has been used
successfully in sensor network coverage~\cite{desilva2007coverage}, protein
and DNA structure analysis~\cite{gameiro2015protein, emmett2015chromatin},
materials science~\cite{hiraoka2016amorphous}, robotics and path
planning~\cite{bhattacharya2015persistent}, financial time
series~\cite{ismail2020bitcoin}, and many other
domains~\cite{otter2017roadmap}. 

Clustering is a fundamental problem in
machine learning, and there has been growing interest in whether the
topological information captured by persistence diagrams can drive
clustering~\cite{islambekov2019spacetime, panagopoulos2022clustering}. The
obstacle is that persistence diagrams do not form a linear space, so classical clustering algorithms cannot be applied directly.

Two lines of prior work are the starting point for this paper. Marchese et
al.~\cite{marchese2017kmeans} perform $k$-means directly in the space of persistence diagrams, using Wasserstein distances and Fréchet means as cluster centroids. Cao et al.~\cite{cao2025kmeans} compare
$k$-means across several representations of persistence diagrams,
including persistence landscapes, persistence images, persistence
measures, and the diagrams themselves. Both approaches yield mixed
results across datasets and representations. 

The space of persistence diagrams is a metric
space, not a Hilbert space. The standard $k$-means update, which computes
centroids as arithmetic means, is therefore unavailable, and clustering
directly in diagram space requires replacing means with Fr\'echet means, which is expensive because every iteration recomputes Fr\'echet means, each itself an optimisation over
matchings. 

The alternative is to embed diagrams into a Hilbert space,
where the standard update applies. There are two ways to do this. One
maps diagrams into a Hilbert space explicitly through a vectorisation; the other defines a positive-definite kernel on diagrams
and works implicitly in the induced reproducing kernel Hilbert space (RKHS). Instead of computing centroids in diagram space or constructing an explicit vector embedding, we define a positive definite kernel directly on persistence diagrams; the kernel implicitly embeds diagrams into a RKHS, and clustering proceeds through kernel evaluations alone. This removes the cost of Fr\'echet means.
Several positive definite kernels have been proposed for persistence diagrams. The persistence scale-space kernel (PSSK) of Reininghaus et al.~\cite{reininghaus2015stable} is built from a heat-diffusion feature map and is proven stable with respect to the $1$-Wasserstein distance. The persistence weighted Gaussian kernel (PWGK) of Kusano et al.~\cite{kusano2016pwg} embeds each diagram as a weighted measure in an RKHS and is proven stable with respect to both the $1$-Wasserstein and the bottleneck distance. Both kernels, however, are established only to be stable. The sliced Wasserstein (SW) kernel of Carri\`ere et al.~\cite{carriere2017sliced} is proven to be both stable and discriminative with respect to the $1$-Wasserstein distance, so its RKHS distance follows the geometry of the diagram distance rather than merely being bounded by it. Carri\`ere et al.~\cite{carriere2017sliced} further show experimentally that it outperforms both PSSK and PWGK on several classification benchmarks. These properties make it the natural building block for our method. 

A point cloud produces one diagram for each $q$-homology considered. Any clustering method based on persistence diagrams must therefore decide how to combine information across all $q$-homologies. Convex combinations of base kernels are the most widely studied family in kernel learning~\cite{cortes2012alignment}.

Our contributions are as follows.
\begin{enumerate}[itemsep=2pt]
  \item We formulate kernel $k$-means with a convex
        combination of SW kernels, one for each $q$-homology, prove that the kernel is positive definite, and its feature-space metric is stable and discriminative with respect to $1$-Wasserstein distance.
        (Section~\ref{sec:method}).
  \item The method outperforms the baselines on two benchmark datasets and remains competitive on a third synthetic dataset, while being computationally efficient (Section~\ref{sec:experiments}).
  \item We compare the convex combination against two natural alternatives and identify the limitation of each. A single kernel on the union of all homology groups cannot assign separate bandwidths to each $q$-homology kernel, so it cannot weight their relative discriminative contributions — the flexibility that the convex combination is designed to provide. An SW kernel computed directly on the point clouds clusters well when the clouds share a common orientation, but it is not rotation-invariant: accuracy degrades once the clouds are independently rotated or perturbed by noise. It also scales with the number of points in each cloud.
  \item The convex combination assigns a weight to each $q$-homology kernel, which raises a separate question: does the selected weight identify the discriminating homology? The benchmark datasets cannot answer this, for two reasons: the weights are selected by maximising clustering accuracy over a grid, and the accuracy-maximising weight need not concentrate on the discriminating homology and none of the benchmarks has a single homology known to separate the classes, so there is no reference against which the selected weights can be checked. We therefore construct synthetic datasets with a planted discriminating homology and combine them with a deterministic, alignment-based selection criterion, showing that the selected weight does concentrate on the discriminating homology kernel.(Section~\ref{sec:weightsig}).
\end{enumerate}

\section{Background and Preliminaries}
\label{sec:background}
Persistent homology summarizes the topology of data across scales through a
filtration. A filtration is a nested family of simplicial complexes $\{K_\varepsilon\}_{\varepsilon \in \R}$ indexed by a
scale parameter, with $K_\varepsilon \subseteq K_{\varepsilon'}$ whenever
$\varepsilon \le \varepsilon'$. Applying homology with coefficients in a field to a
filtration yields, for each dimension $q \ge 0$, a family of vector spaces
$H_q(K_\varepsilon)$ connected by the linear maps induced by the inclusions
$K_\varepsilon \hookrightarrow K_{\varepsilon'}$; such a family is called persistence
module. The dimension of $H_q(K_\varepsilon)$ counts the $q$-dimensional
features present at scale $\varepsilon$: connected components for $q=0$,
independent loops for $q=1$, enclosed voids for $q=2$.

By the structure theorem for persistence modules~\cite{chazal2016structure}, any finitely generated persistence module over a field can be decomposed uniquely into a direct sum of interval
modules. Each interval corresponds to a topological feature that is
born at some scale $b$ and dies at some scale $d > b$. The multiset of pairs
$(b,d)$ is the persistence diagram of the filtration in dimension $q$, denoted
$D_q$. Features with long lifespans $d-b$ are regarded as topological
signal, those with short lifespans as noise.

\subsection{Distance between Persistence Diagrams}
\label{sec:distances}

A persistence diagram $D$ is a multiset of points $\{(x,y)\mid x>0,y>0\}\setminus \Delta$, where $\Delta=\{(x,x)\mid x>0\}$ is called the diagonal.

\begin{definition}[Diagram distance]
\label{def:dp}
Let $p \in \N$ and let $D_1, D_2$ be two diagrams. A \emph{partial matching}
between $D_1$ and $D_2$ is a bijection $\gamma : D_1 \supseteq A \to B
\subseteq D_2$ between a subset $A$ of $D_1$ and a subset $B$ of
$D_2$. For $x \in A$, the cost of matching is defined as
\[
  c(x) \;:=\; \lVert x - \gamma(x) \rVert_\infty^{\,p}.
\]
For $y \in (D_1 \cup D_2) \setminus (A \cup B)$, the cost of matching is defined as 
\[
  c'(y) \;:=\; \lVert y - \pi_\Delta(y) \rVert_\infty^{\,p},
\]
where $\pi_\Delta$ denotes the orthogonal projection onto $\Delta$. The cost $c(\gamma)$ is defined as
\[
  c(\gamma) \;:=\; \Bigl(\, \sum_{x} c(x)
  \;+ \sum_{y} c'(y) \Bigr)^{1/p},
\]
and the \emph{$p$th diagram distance} is defined as
\[
  d_p(D_1, D_2) \;:=\; \inf_\gamma\, c(\gamma).
\]
For $p = \infty$, the cost $c(\gamma)$ is defined as
\[
  c(\gamma) \;:=\; \max\Bigl\{ \sup_{x \in A} \lVert x - \gamma(x)
  \rVert_\infty,\; \sup_{y} \lVert y - \pi_\Delta(y) \rVert_\infty \Bigr\},
\]
and the distance $d_\infty(D_1,D_2) := \inf_\gamma c(\gamma)$ is
called the \emph{bottleneck distance}.
\end{definition}
\subsection{Kernel Methods and Gaussian Kernel}
\label{sec:kernels}

\begin{definition}[Positive definite kernel]
Given a set $\mathcal{X}$, a symmetric function
$k : \mathcal{X}\times\mathcal{X} \to \R$ is a \emph{positive definite kernel}
if for all $n \in \N$, all $x_1,\dots,x_n \in \mathcal{X}$ and all
$a_1,\dots,a_n \in \R$,
$\sum_{i,j=1}^{n} a_i a_j\, k(x_i, x_j) \ge 0$.
\end{definition}

Given a positive definite kernel $k$ on $\mathcal{X}$ there exist a reproducing
kernel Hilbert space (RKHS) $\mathcal{H}_k$ and a feature map
$\phi : \mathcal{X} \to \mathcal{H}_k$ with
$k(x,y) = \inner{\phi(x)}{\phi(y)}_{\mathcal{H}_k}$~\cite{aronszajn1950theory}. The kernel induces a
metric $d_k$ on $\mathcal{X}$, computable from kernel evaluations alone:
\begin{equation}
  d_k^2(x,y) \;=\; k(x,x) - 2k(x,y) + k(y,y).
  \label{eq:dk}
\end{equation}

\begin{definition}[Conditionally negative definite function]
A symmetric function $f : \mathcal{X}\times\mathcal{X} \to \R$ is
\emph{conditionally negative definite} (CND) if for all $n \in \N$, all
$x_1,\dots,x_n \in \mathcal{X}$ and all $a_1,\dots,a_n \in \R$ with
$\sum_i a_i = 0$,
$\sum_{i,j=1}^{n} a_i a_j\, f(x_i, x_j) \le 0$.
\end{definition}

\begin{definition}[Gaussian kernel]
\label{def:rbf}
Let $f : \mathcal{X}\times\mathcal{X} \to [0,\infty)$ be conditionally negative
definite. For $\sigma > 0$, the Gaussian kernel associated with $f$ is
\[
  k_\sigma(x,y) \;=\; \exp\Bigl(-\frac{f(x,y)}{2\sigma^2}\Bigr).
\]
By~\cite[Theorem~3.2.2]{berg1984harmonic}, $k_\sigma$ is positive definite for
every $\sigma > 0$ if and only if $f$ is conditionally negative definite.
\end{definition}

The CND requirement is the obstruction to building Gaussian kernels directly
from diagram distances. Reininghaus et
al.~\cite[Appendix~A]{reininghaus2014stable} observe experimentally that $d_p$
fails to be CND for several values of $p$, including $p=\infty$. 

In this article we use the Sliced Wasserstein distance, introduced by Rabin et al.~\cite{rabin2011wasserstein} and adapted to persistence diagrams by Carrière et al.~\cite{carriere2017sliced}, an approximation of $d_1$ that circumvents this obstruction by reducing the comparison of planar diagrams to a family of one-dimensional comparisons. The construction, given in the next section, rests
on a single classical fact from optimal transport: on the real line, the
$1$-Wasserstein distance between measures coincides with an $L^1$ distance, and it is therefore conditionally negative
definite~\cite[Lemma~3.2]{carriere2017sliced}.

Following~\cite{carriere2017sliced}, we work with non-negative, not
necessarily normalised, measures on $\R$ of equal total mass. For such measures $\mu, \nu$ with $\mu(\R) = \nu(\R) = r$, the
$1$-Wasserstein distance $\mathcal{W}_1(\mu,\nu)$ admits an $L^1$
representation on the real line. For
probability measures this is
classical~\cite[Prop.~2.17]{santambrogio2015optimal}; the extension to
unnormalised measures of equal mass is established
in~\cite[Prop.~2.1]{carriere2017sliced}. This $L^1$ form is precisely
what makes $\mathcal{W}_1$ conditionally negative definite, and hence a
valid basis for a positive-definite kernel. 

A good kernel on persistence diagrams should be both stable and discriminative:diagrams close in diagram distance should map to similar representations, while the diagrams that are far apart should map to distinguishable ones. 

\subsection{The Sliced Wasserstein Kernel}
\label{sec:swkernel}

\begin{definition}[Sliced Wasserstein distance, \cite{carriere2017sliced}]
\label{def:sw}
Let $\mathcal{D}$ denote the space of persistence diagrams with at most
countably many points, and $\mathcal{D}^b_f \subset \mathcal{D}$ the subspace
of finite, bounded diagrams. For
$\theta \in \R^2$ with $\|\theta\|_2 = 1$, let
$\pi_\theta : \R^2 \to \R$, $\pi_\theta(p) = \inner{p}{\theta}$, be the
orthogonal projection onto the line spanned by $\theta$, and let
$\pi_\Delta : \R^2 \to \R^2$ be the orthogonal projection onto the diagonal
$\Delta$. Let $D_1, D_2 \in \mathcal{D}^b_f$. For $i=1,2$, define
\[
  \mu_i^\theta = \sum_{p \in D_i} \delta_{\pi_\theta(p)},
  \qquad
  \mu_{i\Delta}^\theta = \sum_{p \in D_i} \delta_{\pi_\theta(\pi_\Delta(p))}.
\]
The sliced Wasserstein distance~\cite{carriere2017sliced} between $D_1$ and $D_2$ is
\begin{equation}
  \SW(D_1, D_2) \;=\; \frac{1}{\pi} \int_{-\pi/2}^{\pi/2}
  \mathcal{W}_1\bigl(\mu_1^\theta + \mu_{2\Delta}^\theta,\;
           \mu_2^\theta + \mu_{1\Delta}^\theta\bigr)\, d\theta .
  \label{eq:sw}
\end{equation}
\end{definition}

Each integrand in \eqref{eq:sw} is a one-dimensional $\mathcal{W}_1$ distance, and
$\SW$ is conditionally negative definite~\cite[Lemma~3.2]{carriere2017sliced},
so Definition~\ref{def:rbf} applies.

\begin{definition}[Sliced Wasserstein kernel]
\label{def:swkernel}
Let $\sigma > 0$ be given. The SW kernel $k_{\SW}:\mathcal{D}^b_f\times \mathcal{D}^b_f\to \R$ is defined as
\begin{equation}
  k_{\SW}(D_1, D_2) \;=\; \exp\Bigl(-\frac{\SW(D_1,D_2)}{2\sigma^2}\Bigr).
  \label{eq:swkernel}
\end{equation}
\end{definition}

Carri\`ere et al.~\cite{carriere2017sliced} prove that on diagrams with at most
$N$ points, $\SW$ is equivalent to the first diagram distance $d_1$;
\begin{equation}
    \frac{1}{2M}\, d_1(D_1,D_2) \;\le\; \SW(D_1,D_2) \;\le\; 2\sqrt{2}\,
  d_1(D_1,D_2), \qquad M = 1 + 2N(2N-1).
  \label{eq:swequiv}
\end{equation}

The stability of $k_{\SW}$ is established with respect to the diagram distance
$d_1$. The persistence scale-space kernel of Reininghaus et
al.~\cite{reininghaus2015stable} is stable with respect to the $1$-Wasserstein
distance, but they show that a kernel which is additive,
$k(D_1 \cup D_2, D_3) = k(D_1, D_3) + k(D_2, D_3)$ for all
$D_1, D_2, D_3 \in \mathcal{D}$, and non-trivial, i.e.\ $k(D_1, D_2) \neq 0$ for some $D_1, D_2 \in \mathcal{D}$, cannot be stable with respect to the $p$-Wasserstein distance for any $p > 1$, in particular against $d_\infty$.
The persistence weighted Gaussian kernel of Kusano et al.~\cite{kusano2016pwg}
attains stability with respect to $d_\infty$, and moreover, by composing with the bound
$d_\infty(D_q(X), D_q(Y)) \le d_H(X, Y)$ on the Hausdorff distance, data-level
stability. In the discriminativity direction,
however, neither kernel admits a lower bound of the above form, so to our
knowledge $k_{\SW}$ is the only kernel for persistence diagrams proven to be both stable and discriminative.

\subsection{Kernel $k$-means}
\label{sec:kkm}

Classical $k$-means partitions a dataset $\{x_1,\dots,x_N\} \subset \mathcal{X}$
into $C$ clusters $\{\pi_1,\dots,\pi_K\}$ by minimising
\[
  \mathcal{D}(\{\pi_j\}) = \sum_{j=1}^{K} \sum_{x \in \pi_j} \|x - c_j\|^2,
  \qquad
  c_j = \frac{1}{|\pi_j|} \sum_{x \in \pi_j} x .
\]
Kernel $k$-means~\cite{girolami2002mercer,dhillon2004kernel} replaces each $x$ by its image $\phi(x) \in \mathcal{H}_k$ under the feature map of a kernel $k$; the objective becomes
\begin{equation}
  \mathcal{D}_k(\{\pi_j\}) = \sum_{j=1}^{K} \sum_{x \in \pi_j}
  \|\phi(x) - m_j\|^2_{\mathcal{H}_k},
  \qquad
  m_j = \frac{1}{|\pi_j|} \sum_{x \in \pi_j} \phi(x).
  \label{eq:kkm}
\end{equation}
The centroid $m_j$ lives in $\mathcal{H}_k$ and generally has no pre-image in
$\mathcal{X}$, but distances to it require only kernel evaluations:

\begin{proposition}[Distance in feature space]
\label{prop:dist}
For any $x_n \in \mathcal{X}$ and cluster $\pi_j$ with centroid
$m_j = \frac{1}{|\pi_j|}\sum_{m \in \pi_j} \phi(x_m)$,
\begin{equation}
  \|\phi(x_n) - m_j\|^2_{\mathcal{H}_k}
  = k(x_n,x_n)
  - \frac{2}{|\pi_j|} \sum_{m \in \pi_j} k(x_n, x_m)
  + \frac{1}{|\pi_j|^2} \sum_{m,\ell \in \pi_j} k(x_m, x_\ell).
  \label{eq:distformula}
\end{equation}
\end{proposition}
\begin{proof}
Expanding the squared norm and applying the reproducing property
$\langle \phi(x), \phi(y) \rangle_{\mathcal{H}_k} = k(x,y)$ termwise,
\begin{align*}
  \|\phi(x_n) - m_j\|_{\mathcal{H}_k}^2
  &= \langle \phi(x_n), \phi(x_n) \rangle
     - 2\langle \phi(x_n), m_j \rangle
     + \langle m_j, m_j \rangle \\
  &= k(x_n,x_n)
     - \frac{2}{|\pi_j|} \sum_{m \in \pi_j} k(x_n, x_m)
     + \frac{1}{|\pi_j|^2} \sum_{m,\ell \in \pi_j} k(x_m, x_\ell),
\end{align*}
where the last two terms follow by substituting
$m_j = \tfrac{1}{|\pi_j|}\sum_{m\in\pi_j}\phi(x_m)$ and using bilinearity.
\end{proof}
Equation~\eqref{eq:distformula} expresses every point-to-centroid distance purely
in terms of entries of the Gram matrix $K$, so the centroids $m_j$ never need to
be represented explicitly in $\mathcal{H}_k$. This yields a Lloyd-style iteration
that alternates between assigning each point to its nearest centroid and updating
the clusters, with all computation carried out through kernel evaluations
(Algorithm~\ref{alg:kkm}).
\begin{algorithm}[H]
\caption{Kernel $k$-means}
\label{alg:kkm}
\begin{algorithmic}[1]
\Require Data $\{x_1,\dots,x_N\}$, kernel $k$, number of clusters $C$.
\State Compute the Gram matrix $K \in \R^{N \times N}$, $K_{nm} = k(x_n,x_m)$.
\State Initialise clusters $\{\pi_j^{(0)}\}_{j=1}^C$ via $k$-means++.
\Repeat
  \For{each point $x_n$ and each cluster $j$}
    \State Compute $d^2_{jn}$ by \eqref{eq:distformula}.
  \EndFor
  \State Reassign each $x_n$ to $j^\star(x_n) = \arg\min_j d^2_{jn}$.
  \State Update clusters $\pi_j^{(t+1)} = \{x_n : j^\star(x_n) = j\}$.
\Until{cluster assignments stop changing.}
\State \Return final partition $\{\pi_j\}$.
\end{algorithmic}
\end{algorithm}

\section{Sliced Wasserstein Kernel $k$-means}
\label{sec:method}
For each homology $q \in \{0,1,\dots,Q\}$, let
$D^{(q)}_n$ denote the restriction of the persistence diagram $D_n$ to homology $q$, i.e.\ the multiset of birth--death points arising
from the $q$-th homology group $H_q$. Let $k_{H_q}$ denote the sliced Wasserstein kernel of
Definition~\ref{def:swkernel} applied to homology $q$ diagrams:
\[
  k_{H_q}(D_n, D_m) =
  \exp\Bigl(-\frac{\SW\bigl(D^{(q)}_n, D^{(q)}_m\bigr)}{2\sigma_q^2}\Bigr),
\]
where $\sigma_q > 0$ is a bandwidth chosen separately for each homology.
Each $k_{H_q}$ measures topological similarity restricted to a single
homology $q$.
\begin{definition}[Convex combination of SW kernels]
\label{def:convex}
For $Q \in \N\cup \{0\}$ 
\[
  \mathcal{D}_b^f(Q) \;=\;
  \bigl\{\, D \in \mathcal{D}_b^f \;\big|\; H_k(D) = \{0\}\ \text{for all } k > Q \,\bigr\},
\] and let
$\lambda = (\lambda_0,\dots,\lambda_Q)$ with $\lambda_q \ge 0$ and
$\sum_{q=0}^{Q} \lambda_q = 1$. The convex combination of sliced
Wasserstein kernels is defined as
\begin{equation}
  k_\lambda(D_n, D_m) \;=\; \sum_{q=0}^{Q} \lambda_q\, k_{H_q}(D_n, D_m).
  \label{eq:convex}
\end{equation}
\end{definition}

\begin{proposition}[Positive definiteness]
\label{prop:pd}
The kernel $k_\lambda$ of \eqref{eq:convex} is positive definite.
\end{proposition}

\begin{proof}
For scalars $a_1,\dots,a_N \in \R$ and diagrams $D_1,\dots,D_N$,
\[
  \sum_{n,m} a_n a_m\, k_\lambda(D_n,D_m)
  = \sum_{q=0}^{Q} \lambda_q \sum_{n,m} a_n a_m\, k_{H_q}(D_n,D_m) \;\ge\; 0,
\]
since each inner sum is non-negative ($k_{H_q}$ is positive definite) and each
$\lambda_q \ge 0$.
\end{proof}
\subsection{The Induced RKHS Distance}
\label{sec:rkhsdist}

We measure the difference between two persistence diagrams in two ways: by
$d_1(D,D')$, the $1$-Wasserstein diagram distance and by
$d_{k}(D,D')$, the RKHS metric induced by a kernel $k$, i.e.\ the distance in
the Hilbert space where kernel $k$-means runs. Stability and discriminativity ensure the two distances agree upto constant factors, so that clustering in the RKHS respects the original distances between diagrams. Recall from \eqref{eq:dk}
that any positive definite kernel $k$ has a feature map $\phi$ with
$d_k(x,y) = \|\phi(x) - \phi(y)\|_{\mathcal{H}_k}$, and that for a Gaussian kernel
$k(x,y)=\exp(-f(x,y)/2\sigma^2)$ one has $k(x,x)=e^0=1$, so
\begin{equation}
  d_k^2(x,y)=2\Bigl(1-\exp\!\bigl(-f(x,y)/2\sigma^2\bigr)\Bigr).
  \label{eq:rbfdist}
\end{equation}

The following properties are established in~\cite{carriere2017sliced}. For each $q$-homology kernel:
\begin{itemize}
  \item[(I1)] \textbf{$\SW$ is conditionally negative definite}, hence
        $k_{\SW}$ is a positive definite kernel for every $\sigma>0$.
  \item[(I2)] \textbf{Stability:} $\SW(D,D')\le 2\sqrt2\,d_1(D,D')$.
  \item[(I3)] \textbf{Discriminativity:} on diagrams with at most $N$ points,
        $\dfrac{1}{2M}\,d_1(D,D')\le \SW(D,D')$, where $M=1+2N(2N-1)$.
\end{itemize}

\begin{lemma}[Metric decomposition]
\label{lem:decomp}
Let $d_{k_\lambda}$ be the RKHS metric induced by $k_\lambda$ and, for each
homology $q$, let $d_{k_{H_q}}$ be the RKHS metric induced by $k_{H_q}$. Then for all diagrams $D_n,D_m$,
\begin{equation}
  d_{k_\lambda}^2(D_n,D_m)=\sum_{q=0}^{Q}\lambda_q\,d_{k_{H_q}}^2\big(D_n^{(q)},D_m^{(q)}\big).
  \label{eq:decomp}
\end{equation}
\end{lemma}
\begin{proof}
Let $(\mathcal{H}_q,\phi_q)$ be the RKHS and feature map of $k_{H_q}$. Consider the space 
$\mathcal{H}_\lambda=\bigoplus\mathcal{H}_q$, in which a vector is a
tuple $(u_0,\dots,u_Q)$ with $u_q\in\mathcal{H}_q$, the components of different
dimensions are mutually orthogonal, and the squared length of a tuple is the
sum of the squared lengths of its components:
\begin{equation}
  \big\|(u_0,\dots,u_Q)\big\|_{\mathcal{H}_\lambda}^2
  =\sum_{q=0}^{Q}\|u_q\|_{\mathcal{H}_q}^2.
  \label{eq:pythagoras}
\end{equation}
Define the combined feature map
\begin{equation}
  \Phi_\lambda(D)=\Bigl(\sqrt{\lambda_0}\,\phi_0(D),\,\sqrt{\lambda_1}\,\phi_1(D),\,\dots,\,\sqrt{\lambda_Q}\,\phi_Q(D)\Bigr).
  \label{eq:featuremap}
\end{equation}
Then
\[
  \inner{\Phi_\lambda(D_n)}{\Phi_\lambda(D_m)}
  =\sum_{q=0}^{Q}\bigl(\sqrt{\lambda_q}\bigr)^2\inner{\phi_q(D_n)}{\phi_q(D_m)}
  =\sum_{q=0}^{Q}\lambda_q\,k_{H_q}(D_n,D_m)
  =k_\lambda(D_n,D_m),
\]
Hence $d_{k_\lambda}^2=\|\Phi_\lambda(D_n)-\Phi_\lambda(D_m)\|_{\mathcal{H}_\lambda}^2$, and by \eqref{eq:pythagoras},
\[
  d_{k_\lambda}^2
  =\sum_{q}\big\|\sqrt{\lambda_q}\big(\phi_q(D_n^{(q)})-\phi_q(D_m^{(q)})\big)\big\|_{\mathcal{H}_q}^2
  =\sum_{q}\lambda_q\big\|\phi_q(D_n^{(q)})-\phi_q(D_m^{(q)})\big\|_{\mathcal{H}_q}^2
  =\sum_{q}\lambda_q\,d_{k_{H_q}}^2\big(D_n^{(q)},D_m^{(q)}\big).
\]
\end{proof}
The construction shows more than~\eqref{eq:decomp}: the direct sum
$\mathcal{H}_\lambda=\bigoplus\mathcal{H}_q$, with inner product
$\inner{(u_q)_q}{(v_q)_q}_{\mathcal{H}_\lambda}=\sum_{q=0}^{Q}\inner{u_q}{v_q}_{\mathcal{H}_q}$,
is itself an RKHS, and $(\mathcal{H}_\lambda,\Phi_\lambda)$ is a
feature-space representation of
$k_\lambda=\sum_{q=0}^{Q}\lambda_q\,k_{H_q}$, since
$\inner{\Phi_\lambda(D_n)}{\Phi_\lambda(D_m)}_{\mathcal{H}_\lambda}
=k_\lambda(D_n,D_m)$.

\subsection{Stability and Disriminativity of the Convex Combination}
\label{sec:stability}
We now give the main theoretical result of this article, which states that the  $d_{k_\lambda}$ preserves the metric between persistence diagrams upto constants, which should lead to faithful clustering. This is confirmed in Section~\ref{sec:experiments}, where we report improved clustering accuracies on benchmark datasets.
\begin{theorem}[Stability of $k_\lambda$]
\label{thm:stab}
Let $\sigma_{\min}=\min_q\sigma_q>0$. Then
\[
  d_{k_\lambda}^2(D_n,D_m)\ \le\ \frac{2\sqrt2}{\sigma_{\min}^2}\sum_{q=0}^{Q}\lambda_q\,d_1^{(q)}\big(D_n^{(q)},D_m^{(q)}\big).
\]
If the diagrams are close in the true distance $d_1$ in every
dimension, they are close in the combined RKHS.
\end{theorem}
\begin{proof}
Fix a homology $q$ and write $\SW^{(q)}=\SW^{(q)}(D_n^{(q)},D_m^{(q)})$ and $d_1^{(q)}=d_1^{(q)}(D_n^{(q)},D_m^{(q)})$. We first bound the each $q$-homology term:
\begin{equation}
  d_{k_{H_q}}^2
  = 2\Bigl(1-e^{-\SW^{(q)}/2\sigma_q^2}\Bigr)
  \le 2\cdot\frac{\SW^{(q)}}{2\sigma_q^2} \le \frac{2\sqrt2}{\sigma_q^2}\,d_1^{(q)}.
  \label{eq:perdim-bound}
\end{equation}
Using Lemma~\ref{lem:decomp} and~\eqref{eq:perdim-bound}, we get
\[
  d_{k_\lambda}^2
  =\sum_q\lambda_q\,d_{k_{H_q}}^2
  \le\sum_q\lambda_q\,\frac{2\sqrt2}{\sigma_q^2}\,d_1^{(q)}
  \le\frac{2\sqrt2}{\sigma_{\min}^2}\sum_q\lambda_q\,d_1^{(q)}. \qedhere
\]
\end{proof}
\begin{theorem}[Discriminativity of $k_\lambda$]
\label{thm:disc}
For each homology $q$, let $T_q>0$ be such that $\SW^{(q)}/2\sigma_q^2\le T_q$. For $M=1+2N(2N-1)$ define
\[
  \kappa_q=\frac{1-e^{-T_q}}{2\sigma_q^2\,T_q\,M} > 0.
\]
Then
\[
  d_{k_\lambda}^2(D_n,D_m)\ \ge\ \Bigl(\min_q\kappa_q\Bigr)\sum_{q=0}^{Q}\lambda_q\,d_1^{(q)}\big(D_n^{(q)},D_m^{(q)}\big).
\]
\end{theorem}

\begin{proof}
Fix $q$, write $t=\SW^{(q)}/2\sigma_q^2\in[0,T_q]$. Note that
\[
  d_{k_{H_q}}^2 = 2\bigl(1-e^{-t}\bigr).
\]
Since $\phi(t)=1-e^{-t}$ is concave on $[0,T_q]$, we get
\begin{align}
  d_{k_{H_q}}^2
  &= 2\bigl(1-e^{-t}\bigr)
   \ge 2\,(1-e^{-T_q})\,\frac{t}{T_q}
   =\frac{1-e^{-T_q}}{\sigma_q^2\,T_q}\,\SW^{(q)} \nonumber\\
  &\ge \frac{1-e^{-T_q}}{\sigma_q^2\,T_q}\cdot\frac{1}{2M}\,d_1^{(q)}
   =\kappa_q\,d_1^{(q)}.
   \label{eq:kappa}
\end{align}
Using Lemma~\ref{lem:decomp} and~\eqref{eq:kappa}, we get
\[
  d_{k_\lambda}^2
  =\sum_q\lambda_q\,d_{k_{H_q}}^2
  \ge\sum_q\lambda_q\,\kappa_q\,d_1^{(q)}
  \ge\Bigl(\min_q\kappa_q\Bigr)\sum_q\lambda_q\,d_1^{(q)}. \qedhere
\]
\end{proof}
Theorems~\ref{thm:stab} and~\ref{thm:disc} show that $d_{k_\lambda}$ behaves like $d_1$ distance, so the kernel neither collapses distinct diagrams nor exaggerates small perturbations. This two-sided control is what makes clustering in the RKHS faithful to the underlying topological structure. See Algorithm~\ref{alg:main}.
\begin{algorithm}[H]
\caption{Sliced Wasserstein kernel $k$-means.}
\label{alg:main}
\begin{algorithmic}[1]
\Require Point clouds $\{X_1,\dots,X_N\}$; clusters $C$; max dimension $Q$;
weight grid $\Lambda \subset \Delta^Q$; number of SW directions $L$.
\Statex \textbf{Stage 1 --- Gram matrices.}
\For{$n = 1,\dots,N$}
  \State Compute persistence diagrams $D_n^{(0)},\dots,D_n^{(Q)}$ of $X_n$.
\EndFor
\For{$q = 0,\dots,Q$}
  \State $S^{(q)}_{nm} \gets \SW\bigl(D_n^{(q)}, D_m^{(q)}\bigr)$ for all
         $n < m$, approximated with $L$ directions.
  \State $\sigma_q \gets$ median of the off-diagonal entries of $S^{(q)}$.
  \State $K^{(q)} \gets \exp\bigl(-S^{(q)}/(2\sigma_q^2)\bigr)$, with
         $K^{(q)}_{nn} = 1$.
\EndFor
\Statex \textbf{Stage 2 --- weight selection over the grid.}
\For{each $\lambda \in \Lambda$}
  \State $K_\lambda \gets \sum_q \lambda_q K^{(q)}$.
  \State Run Algorithm~\ref{alg:kkm} on $K_\lambda$ to obtain
         $\hat\pi(\lambda)$.
  \State Score $\hat\pi(\lambda)$: ARI against ground truth when labels are
         available; otherwise the silhouette score under the
         RKHS metric \eqref{eq:dk} of $K_\lambda$.
\EndFor
\Statex \textbf{Stage 3 --- output.}
\State $\lambda^\star \gets$ the best-scoring weight vector.
\State \Return $\hat\pi(\lambda^\star)$ and $\lambda^\star$.
\end{algorithmic}
\end{algorithm}

\section{Experiments}
\label{sec:experiments}
We evaluate the performance of kernel $k$-means on three datasets, the four-class
signal dataset, circle--sphere--torus point-cloud dataset, and the 3D
shape-matching dataset~\cite{sumner2004deformation}. We also compare the convex combination against two natural alternatives. The first is the kernel on union of all homology groups and second is the SW kernel computed directly on point clouds: each cloud is
treated as a uniform discrete measure on its ambient space $\mathbb{R}^d$. We use the sliced Wasserstein distance between two such measures, sampling unit
directions uniformly on the sphere $S^{d-1}$. The convex combination is parameterised by a separate weight for each $q$-homology kernel. The next question we ask---whether the selected weights land on the homology which discriminates the classes. This question cannot be answered by the datasets used for clustering because there is no single discriminating homology. To answer this question, we generated the three datasets with a single known discriminating homology. We use the Adjusted Rand Index (ARI) to evaluate the performance of our clustering algorithm.
 
\subsection{Signal Dataset}
\label{subsec:signal}
We first evaluate our method on the synthetic signal dataset~\cite{marchese2017kmeans} of four types of signals. The four signals are given by the following equations:
\begin{align}
  w_i      &= \omega \sin(o_i) + \eta_i,                                 \label{eq:cls0}\\
  v_{i+1}  &= v_i + \eta_i,                                              \label{eq:cls1}\\
  u_{i+1}  &= \alpha \sin(u_i) + \eta_i,                                 \label{eq:cls2}\\
  z_i      &= \omega\bigl(1 + 0.5\cos o_i\bigr)\cos o_i + \eta_i,        \label{eq:cls3}
\end{align}
where $\omega \sim \mathcal{U}(1,3)$ and
$\alpha = 2.5$. These equations represent,
respectively, a noisy periodic signal~\eqref{eq:cls0}, a random walk~\eqref{eq:cls1}, a bistable time-series ~\eqref{eq:cls2}, and a doubly periodic signal~\eqref{eq:cls3}. Since the
amplitude $\omega$ is randomised in both periodic classes, the classes are separated by the geometry of their trajectories in delay space rather than amplitude. To expose this shape we reconstruct each signal's phase space via Takens'
delay embedding~\cite{takens1980}. Given a scalar time series
$\{x_i\}_{i=0}^{N-1}$, a delay $\tau \in \mathbb{R}^{+}$, and an embedding
dimension $d \in \mathbb{Z}^{+}$, the embedding maps each admissible index
$i$ to the vector
\begin{equation}
  \mathbf{y}_i
    = \bigl(x_i,\; x_{i+\tau},\; x_{i+2\tau},\;\dots,\;
            x_{i+(d-1)\tau}\bigr) \in \mathbb{R}^d ,
  \qquad i = 0,\dots,N-(d-1)\tau-1 .
\end{equation}
Takens' theorem~\cite{takens1980} states that, for a suitable embedding, this reconstruction is diffeomorphic to the
underlying dynamical system's attractor, so the topological properties of the original system are visible in the reconstructed point cloud.

Table~\ref{tab:signal-results} reports the performance of our method under five settings: The SW kernel computed directly on the embedded point clouds (with and without a random rotation applied to each cloud), the SW kernel on the union of homology groups $(H_0\cup H_1)$, the convex combination of $H_0$ and $H_1$, and the convex combination of the $H_0$, $H_1$ and SW kernel on rotated clouds. The cloud kernel alone clusters the unrotated
clouds essentially perfectly, since the four classes have distinct delay-space shapes; however, because the sliced
Wasserstein distance compares clouds along a fixed set of directions, it
is not rotation invariant, and a random rotation of each cloud reduces
the ARI. Kernel on union of $H_0$ and $H_1$ performs poorly compared to the convex combination of $H_0$ and $H_1$, which outperforms the baseline~\cite{marchese2017kmeans}. Supplementing it with the cloud kernel further reduces the mean error indicating that geometric and topological information
are complementary.

\begin{table}[H]
\centering
\begin{tabular}{lcc}
\toprule
Experimental setting & Mean ARI $\uparrow$ & Mean error $\downarrow$ \\
\midrule
SW kernel on cloud                          & 0.9999 & 0.0000 \\
SW kernel on randomly rotated cloud        & 0.6102 & 0.2832 \\
$H_0 \cup H_1$              & 0.4139 & 0.3852 \\
$\lambda_0H_0 + \lambda_1H_1$                       & \textbf{0.6052} & \textbf{0.2012} \\
$\lambda_0H_0 + \lambda_1H_1 + \lambda_2 \text {rotated cloud}$     & \textbf{0.7795} & \textbf{0.1103} \\
\midrule
$k$-Fr\'echet means baseline~\cite{marchese2017kmeans} & --- & 0.3183 \\
\bottomrule
\end{tabular}
\caption{Clustering performance on the four-class signal dataset.}
\label{tab:signal-results}
\end{table}

\subsection{Circles, Spheres, and Torus Point Clouds}
\label{subsec:cst}
We next evaluate on the synthetic point-cloud benchmark~\cite{cao2025kmeans}. We perturb each coordinate of every point by i.i.d.\ noise drawn uniformly from $s\cdot[-0.5,0.5]$, where the noise level $s$ ranges over $\{1,2,3,4,5,10\}$. Table~\ref{tab:cst-baselines} reports the mean ARI of our three settings---the SW kernel on the point clouds, the SW kernel on the union of homology groups $(H_0\cup H_1\cup H_2)$, and the convex combination of $q$-homology kernels alongside the persistence-based embeddings benchmarked in~\cite{cao2025kmeans}. We compare against the four representations benchmarked
in~\cite{cao2025kmeans}: persistence landscapes
(PL)~\cite{bubenik2015}, persistence images
(PI)~\cite{adams2017}, persistence measures
(PM)~\cite{divol2021} and persistence diagrams (PD) themselves. The cloud SW kernel is the most sensitive to noise, since coordinate-wise
perturbations directly distort the geometry. The
persistence-based kernels degrade far more slowly. The convex
combination improves on the kernel on the union of homology groups at every noise level (Figure~\ref{fig:cst-noise}).

\begin{table}[H]
\centering
\begin{tabular}{lccccccc}
\toprule
Noise $s$ & PL & PI & PD & PM & Cloud & $H_0\cup H_1\cup H_2$& $\lambda_0H_0+\lambda_1H_1+\lambda_2H_2$ \\
\midrule
1  & 0.998 & 0.892 & 0.971 & 0.996 & 0.978 & 0.994 & 0.996 \\
2  & 0.974 & 0.965 & 0.952 & 0.961 & 0.863 & 0.934 & 0.941 \\
3  & 0.911 & 0.922 & 0.892 & 0.947 & 0.757 & 0.922 & 0.928 \\
4  & 0.890 & 0.852 & 0.904 & 0.946 & 0.732 & 0.927 & 0.932 \\
5  & 0.843 & 0.734 & 0.890 & 0.945 & 0.733 & 0.917 & 0.931 \\
10 & 0.633 & 0.453 & 0.854 & 0.892 & 0.685 & 0.866 & 0.889 \\
\bottomrule
\end{tabular}
\caption{Mean ARI on synthetic circle, sphere, and torus dataset.}
\label{tab:cst-baselines}
\end{table}
\begin{figure}[H]
    \centering
    \includegraphics[width=0.62\textwidth]{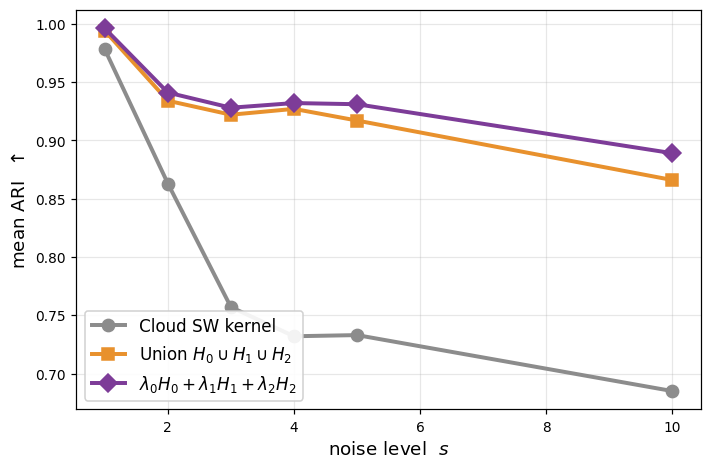}
    \caption{Mean ARI against noise level $s$ on the circle--sphere--torus dataset}
    \label{fig:cst-noise}
\end{figure}

\subsection{3D Shape Matching Data}
\label{subsec:shapedata}
We now turn to the benchmark dataset from 3D shape
matching~\cite{sumner2004deformation} used in~\cite{cao2025kmeans}, where each
mesh is reduced to the point cloud of its vertices and clustered through
persistent homology. The dataset has eight classes; camel, cat, elephant,
face, flamingo, head, horse, lion. Each a single object in several poses,
giving $83$ meshes in total. Table~\ref{tab:shape-results} reports ARI across the nine weights and six
cluster counts. Two patterns stand out. The pure $H_0$ kernel is essentially
uninformative. The
discriminative signal lives in the $H_2$. The pure $H_2$ kernel improves monotonically with $k$ and recovers the
eight classes exactly at $k=8$, where each class forms its own cluster (Figure~\ref{fig:shape-ari}). The weights suggest that the discriminating homology is $H_2$. The convex combination therefore works best when it concentrates on the homology that discriminates the classes, not when it combines the three kernels uniformly. We study this next.

\begin{table}[H]
\centering
\begin{tabular}{l cccccc}
\toprule
Weight $(\lambda_{H_0},\lambda_{H_1},\lambda_{H_2})$ & $k{=}3$ & $4$ & $5$ & $6$ & $7$ & $8$ \\
\midrule
$(1,0,0)$              & 0.018 & 0.016 & 0.015 & 0.014 & 0.013 & 0.014 \\
$(0,1,0)$              & 0.159 & 0.291 & 0.585 & 0.554 & 0.544 & 0.539 \\
$(0,0,1)$              & 0.322 & 0.404 & 0.593 & 0.754 & 0.846 & \textbf{1.000} \\
$(0.5,0.5,0)$ & 0.090 & 0.189 & 0.166 & 0.162 & 0.157 & 0.155 \\
$(0.5,0,0.5)$ & 0.188 & 0.321 & 0.478 & 0.544 & 0.469 & 0.653 \\
$(0,0.5,0.5)$ & 0.159 & 0.408 & 0.572 & 0.659 & 0.870 & 0.909 \\
$(0.3,0.3,0.4)$        & 0.159 & 0.408 & 0.572 & 0.754 & 0.643 & 0.835 \\
$(0.4,0.3,0.3)$        & 0.100 & 0.317 & 0.334 & 0.306 & 0.724 & 0.692 \\
$(0.3,0.4,0.3)$        & 0.159 & 0.408 & 0.568 & 0.655 & 0.870 & 0.869 \\

\bottomrule
\end{tabular}
\caption{ARI on the 3D shape-matching dataset.}
\label{tab:shape-results}
\end{table}
\begin{figure}[H]
\centering
\includegraphics[width=0.62\textwidth]{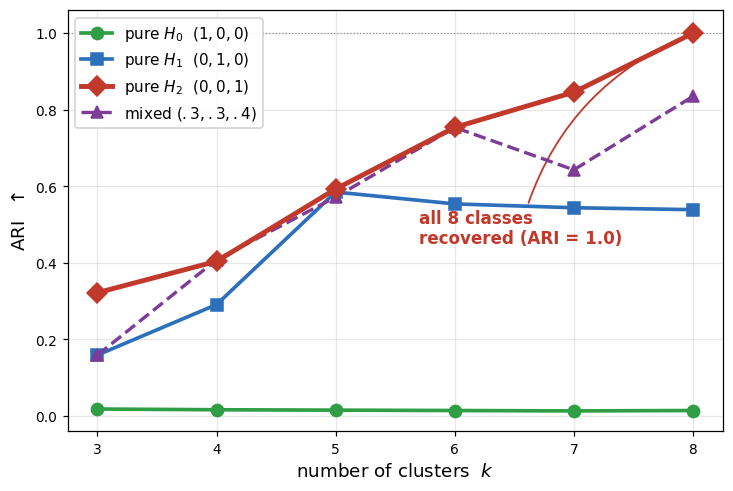}
\caption{ARI against the number of clusters $k$ on the 3D shape-matching
dataset.}
\label{fig:shape-ari}
\end{figure}
\subsection{The Weight-Significance Experiment}
\label{sec:weightsig}

The convex combination $k_\lambda$ uses a separate weight for each $q$-homology kernel, which raises a natural question: does the largest selected weight fall on the kernel corresponding to the homology which discriminates the classes? The benchmark datasets of
Sections~\ref{subsec:signal}--\ref{subsec:shapedata} cannot answer this question, because there is no single discriminating homology, so there is no reference against which a selected weight could be checked. Weight chosen by maximising clustering accuracy need not be the weight that
concentrates on the kernel carrying the class signal. We therefore build three two-class synthetic
datasets $\Szero,\Sone,\Stwo$ in which the separating signal is planted in a single chosen homology $q^\star\in\{0,1,2\}$ and
suppressed in the others. We select a weight triple by a
deterministic criterion --- centred kernel--target alignment of Cortes et al.~\cite{cortes2012alignment}, a centred refinement of  the kernel--target alignment of Cristianini et al.~\cite{cristianini2001alignment}, and
test whether the selected weight concentrates on $q^\star$.

\subsubsection{Datasets}
\label{sec:ws-datasets}

We build three two-class datasets $\mathcal{S}_{q^\star}$, $q^\star\in\{0,1,2\}$, each with $n$
point clouds per class. The two classes of $\mathcal{S}_{q^\star}$ are designed to be
topologically identical in every homology except $q^\star$, so that the classifier uses homology $q^\star$ information. 

For $\Szero$, where the signal is planted in $H_0$ (connected components),
class A is two well-separated Gaussian blobs and class B is three; they differ
only in the number of long-lived $H_0$ bars. For $\Sone$, planted signal is in $H_1$
(loops), class A is points on a single circle and class B is points on a
figure-eight; the two classes match in $H_0$ and $H_2$ and differ only in the
number of significant $H_1$ bars.

Planting the signal in $H_2$ (voids) for $\Stwo$ requires care, as the natural
first choices fail to isolate it. The first choice is sphere vs.\ filled ball. These does differ in $H_2$: the sphere encloses a void and the ball does not. But
they also differ in how the points are arranged. The sphere's points lie on a
surface whereas the ball's fill a three-dimensional region, so for the same
number of points the typical spacing between neighbours is not the same.
$H_0$ Persistent homology is sensitive to exactly this spacing, since
it records the scales at which nearby points first join up. The $H_0$ diagrams
of the sphere and the ball already differ even though both clouds are a single
connected piece ($\beta_0=1$); a classifier could therefore separate the classes
using $H_0$ rather than the $H_2$. The second choice, sphere vs.\ torus, fails for
a different reason: the torus has two independent loops ($\beta_1=2$), so $H_1$
discriminates as well. We
instead use sphere $S^2$ vs.\ a punctured sphere, a full $2$-sphere with a
single spherical cap removed. Topologically the latter is a disk,
$\beta=(1,0,0)$, so the void is destroyed. It is however sampled from the same
surface as the full sphere, so the $H_0$ and $H_1$ diagrams are matched as
closely as the construction allows. Figure~\ref{fig:planted-clouds} shows one sample cloud per class, and
Table~\ref{tab:datasets} summarises the three constructions.
\begin{figure}[H]
\centering
\includegraphics[width=\textwidth]{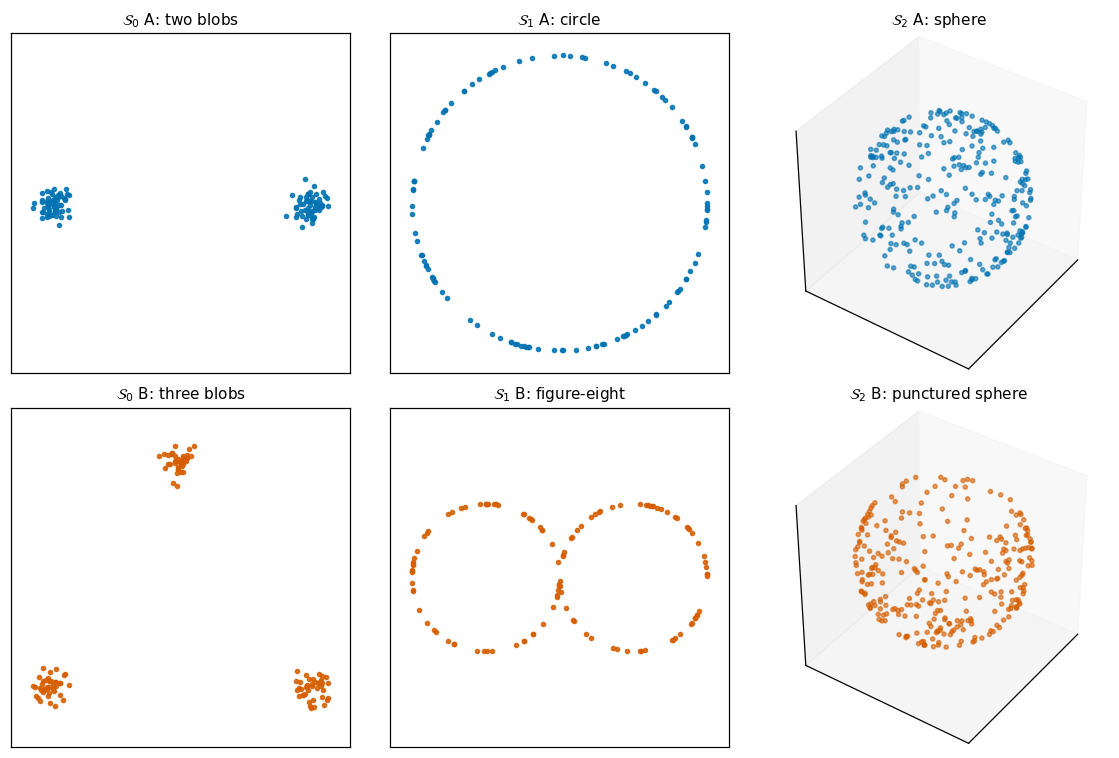}
\caption{One sample cloud per class of the three planted datasets.}
\label{fig:planted-clouds}
\end{figure}
\begin{figure}[H]
\centering
\includegraphics[width=\textwidth]{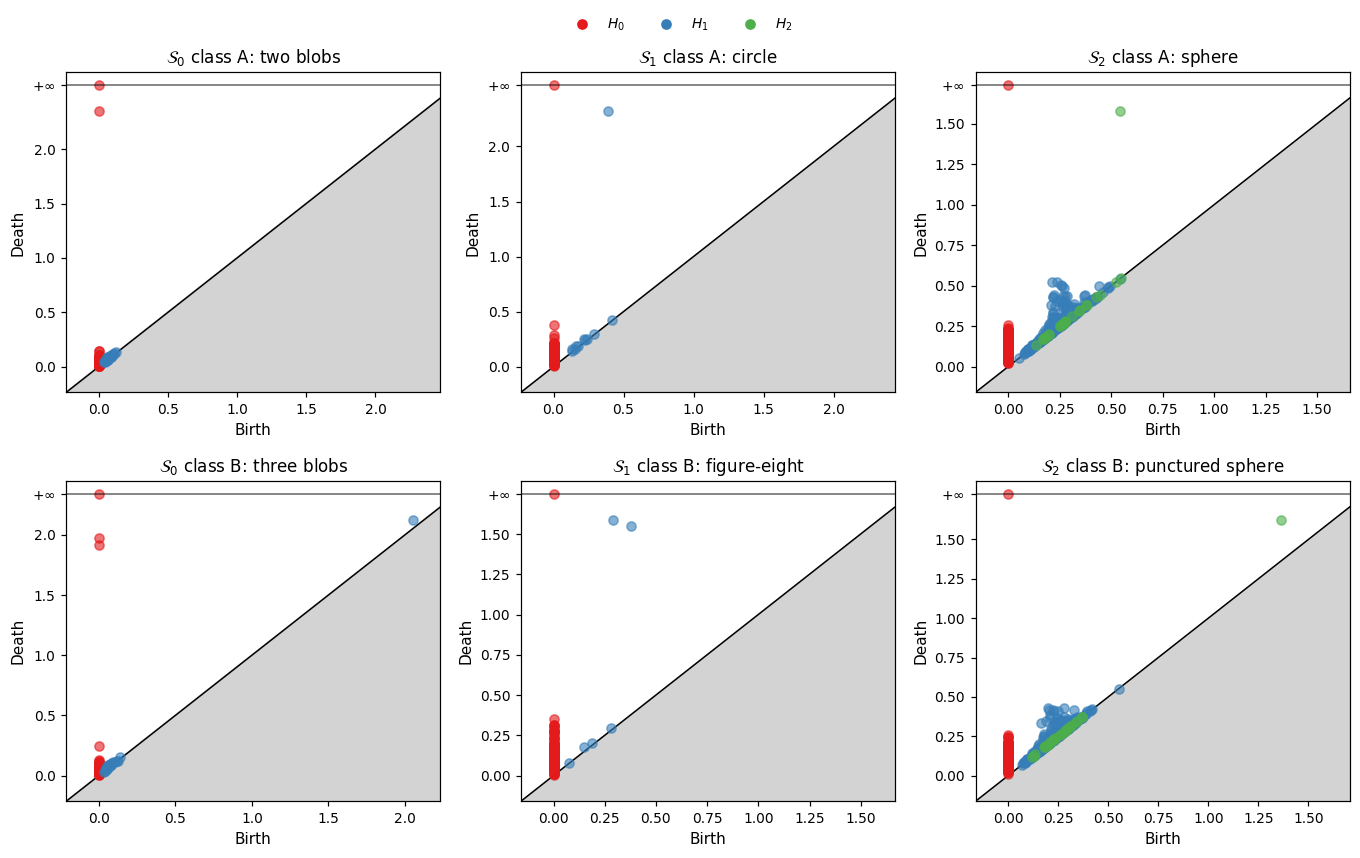}
\caption{Persistence diagrams of the clouds in
Figure~\ref{fig:planted-clouds}.}
\label{fig:planted-diagrams}
\end{figure}

\begin{table}[H]
\centering
\begin{tabular}{l l ccc}
\toprule
Dataset & Classes (A vs.\ B) & $\beta_0$ & $\beta_1$ & $\beta_2$ \\
\midrule
$\Szero$ & two blobs vs.\ three blobs        & \fbox{$2$ vs.\ $3$} & $0$ & $0$ \\
$\Sone$  & circle vs.\ figure-eight          & $1$ & \fbox{$1$ vs.\ $2$} & $0$ \\
$\Stwo$  & sphere vs.\ punctured sphere      & $1$ & $0$ & \fbox{$1$ vs.\ $0$} \\
\bottomrule
\end{tabular}
\caption{Planted homology (boxed)
differs between A and B.}
\label{tab:datasets}
\end{table}

\subsubsection{Preprocessing}
\label{sec:ws-preprocessing}

Three preprocessing choices ensure the discriminative signal lives in the
topology rather than in raw geometry. Each cloud is first divided by a single
scalar, its global standard deviation, which rescales the cloud uniformly and
preserves its shape. Each coordinate is then perturbed by
uniform noise $[-s/2,s/2]$. Figure~\ref{fig:noise-progression} shows the effect of increasing $s$ on a
figure-eight cloud. We normalise first and add noise
second. The three datasets start at very different sizes, the
clusters of $\Szero$ sit tens of units apart, while the circle and sphere of
$\Sone,\Stwo$ have radius about one. Adding the noise first and resizing
afterwards would shrink the noise along with everything else. Finally, after
normalisation, each cloud is randomly rotated, so that the classes can be
separated only by their topology and not by a fixed orientation. For each cloud we compute persistence diagrams $D_q$ for the homology groups $H_0,H_1,H_2$. For the planar families $\Szero,\Sone$ we use the Vietoris--Rips
filtration, whose persistent homology approximates that of the
underlying ball union up to a Čech--Rips interleaving~\cite{edelsbrunner2010computational,chazal2014persistence}. For $\Stwo$ in $\R^3$ we use the alpha complex, which by the Nerve Theorem is homotopy equivalent to the union of balls at every scale~\cite{edelsbrunner2010computational}, so its persistence diagram is exactly that of the ball union. It is also far cheaper than Vietoris--Rips once $H_2$ is required. 

In every cloud we drop the one essential ($\infty$-death) feature
(the $+\infty$ line in Figure~\ref{fig:planted-diagrams}). It carries no class information, since every nonempty cloud has exactly one such feature. The natural fix, capping the death at the largest edge length in the filtration does not help, because that largest length measures the overall extent of the cloud, which differs between the classes. The capped bar would then encode the cloud's size and turn into a discriminative feature of its own. Lastly, we keep only the few most persistent points in each homology (here the top three). This denoising step is justified by the stability theorem~\cite{cohensteiner2005stability}: perturbing a cloud by $\varepsilon$ moves its diagram by
$O(\varepsilon)$, so noise can only produce low-persistence points near the diagonal, while a
genuinely persistent feature stays far from it, as visible in
Figure~\ref{fig:planted-diagrams}. The planted feature is always among the most persistent points, so keeping the high-persistence points keeps the signal and discards the noisy points. The number of points kept must be at least the number of planted features .
\begin{figure}[H]
\centering
\includegraphics[width=0.8\textwidth]{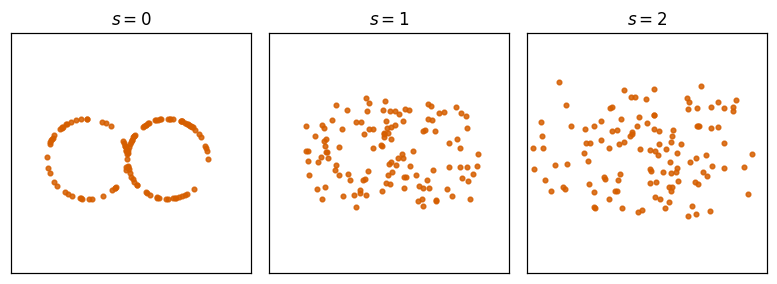}
\caption{A figure-eight cloud from $\Sone$ at noise levels $s = 0, 1, 2$.}
\label{fig:noise-progression}
\end{figure}

\subsubsection{Weight Selection by Centered Kernel--Target Alignment}
\label{sec:ws-alignment}
The weight-selection criterion must satisfy two requirements. First, it must
be deterministic. Second,
it must not be fooled by degenerate kernels.
 
A natural deterministic criterion is the kernel--target alignment of
Cristianini et al.~\cite{cristianini2001alignment}. Let $y$ be the class
labels and $Y_{ij}=\mathds{1}[y_i=y_j]$ the ideal kernel; the (uncentred)
alignment of a kernel matrix $K$ with $Y$ is
\[
  \widehat{A}(K,Y)\;=\;\frac{\inner{K}{Y}_F}{\Fnorm{K}\,\Fnorm{Y}},
\]
which is large when same-class pairs are more similar than different-class
pairs. In our setting, however, the uncentered alignment fails, and it fails
on exactly the degenerate kernels this experiment produces. An empty persistence diagram gives a constant kernel: in $\Szero$ the clouds are
planar, so they enclose no voids and their $H_2$ diagrams are empty; the SW
distance between two empty diagrams is zero, so every entry of the $H_2$
Gram matrix is $\exp(0)=1$, the all-ones matrix $J$. With class sizes
$n_A$ and $n_B$, $n=n_A+n_B$, the entries of $Y$ equal $1$ exactly on
same-class pairs, so $\inner{J}{Y}_F = n_A^2+n_B^2$, $\Fnorm{J}=n$ and
$\Fnorm{Y}=\sqrt{n_A^2+n_B^2}$. Hence
\[
  \widehat{A}(J,Y)\;=\;\frac{\sqrt{n_A^2+n_B^2}}{n}\;\ge\;\frac{1}{\sqrt2},
\]
The uncentred criterion therefore awards a score of at least $0.71$ to a kernel containing no information, and could prefer the empty $H_2$ over
the planted $H_0$. This is an instance of a general phenomenon identified by Cortes et al.~\cite{cortes2012alignment}, who show that uncentered alignment need not correlate with performance and can even correlate negatively. Cortes et al.~\cite{cortes2012alignment} propose centering in feature space as the remedy. We use the
centered kernel--target alignment
\begin{equation}
  A(K,Y)\;=\;\frac{\inner{K_c}{Y_c}_F}{\Fnorm{K_c}\,\Fnorm{Y_c}},
  \qquad K_c = HKH,\quad H = I-\tfrac{1}{n}\mathbf{1}\mathbf{1}^{\!\top}.
  \label{eq:align}
\end{equation}
Centering annihilates the constant part of a kernel. For $J$, we have, $K_c = HJH = 0$.
 
Because $k_\lambda$ is linear in $\lambda$ and centring is linear, with
$K_{c,q}=HK_{H_q}H$ we have
\begin{equation}
  A(k_\lambda,Y)=\frac{\lambda^{\!\top} b}
  {\sqrt{\lambda^{\!\top} G\,\lambda}\;\Fnorm{Y_c}},
  \qquad b_q=\inner{K_{c,q}}{Y_c},\quad
  G_{qj}=\inner{K_{c,q}}{K_{c,j}},
  \quad q,j\in\{0,1,2\},
  \label{eq:bG}
\end{equation}

We precompute $b\in\R^3$ and $G\in\R^{3\times3}$ once per trial, so
evaluating the alignment at any $\lambda$ costs only a dot product and a
quadratic form. The selected weight is
$\lamstar=\argmax_{\lambda\in\Delta^2}A(k_\lambda,Y)$, where the
maximisation is over a fine grid on the simplex $\Delta^2$.

\subsubsection{Results}
\label{sec:ws-results}

For each dataset and noise level $s$ we run $10$ trials, with 10 clouds per class
and 120 points per cloud for $\Szero$ and $\Sone$, and 300 for $\Stwo$, over the noise
grid of Section~\ref{sec:ws-datasets}. In each trial we record the selected weight
$\lamstar$ and the $q$-homology centered alignment $A(K_{H_q},Y)$, and we summarise
the runs by three quantities.

\begin{description}
  \item[(a) Hit rate.] The fraction of the $T$ trials in which the largest selected
  weight falls on the planted homology. Near 1 means the method identifies the correct homology; a uniformly random choice among the three kernels would achieve $1/3$.
  \item[(b) Weight confusion.] For each planted homology $q^\star$ we average the
  selected weight over the $T$ trials and then over the noise levels $s$, giving
  $\bar\lamstar_q(q^\star)$ and the entry $W_{q^\star q}$ of a $3\times3$ matrix $W$.
  Row $q^\star$ indexes the planted homology and column $q$ the homology receiving the weight,
  so $W_{q^\star q}$ is the mean weight placed on $q$-homology kernel when the signal is
  planted in $H_{q^\star}$. We want $W$ strictly diagonally dominant, since correct weight selection means most weight lands on the homology in which the signal is planted.
  \item[(c) Consistency.] For each dataset we check that three homologies coincide:
  the one whose kernel $K_{H_q}$, attains the highest mean alignment;
  the one receiving the largest mean selected weight; and the planted one:
  $\argmax_q \bar A(K_{H_q},Y)=\argmax_q\bar\lamstar_q=q^\star$.
\end{description}
Tables~\ref{tab:S0}--\ref{tab:S2} give the per-noise weights, hit rates, and
single-kernel alignments for the three datasets; the planted homology is
highlighted. Averaging over the noise sweep gives the two $3\times3$ matrices in
Table~\ref{tab:confusion}, with rows indexed by the planted homology and columns
by the homology receiving the weight. Both are strongly diagonally dominant: the
weight-confusion matrix has average diagonal mass $0.95$ and the alignment matrix
$0.85$. The selected weight lands on the planted homology with hit rate $1.00$ up
to moderate noise, the weight-confusion matrix is nearly the identity, and for
every dataset the homology carrying the most weight is also the one with the
highest mean alignment, so criterion (c) holds and the convex combination
identifies the discriminating homology.

\begin{table}[H]
\centering
\begin{tabular}{c ccc c ccc}
\toprule
& \multicolumn{3}{c}{selected weight $\lamstar$} & & \multicolumn{3}{c}{alignment $A(K_{H_q},Y)$}\\
\cmidrule(lr){2-4}\cmidrule(lr){6-8}
$s$ & \hgl{$H_0$} & $H_1$ & $H_2$ & hit & \hgl{$H_0$} & $H_1$ & $H_2$\\
\midrule
$0$    & \hgl{0.865} & 0.000 & 0.135 & 1.00 & \hgl{1.000} & 0.229 & 0.067\\
$0.25$ & \hgl{0.952} & 0.000 & 0.048 & 1.00 & \hgl{1.000} & 0.229 & 0.081\\
$0.5$  & \hgl{1.000} & 0.000 & 0.000 & 1.00 & \hgl{0.998} & 0.244 & 0.087\\
$0.75$ & \hgl{1.000} & 0.000 & 0.000 & 1.00 & \hgl{0.996} & 0.282 & 0.077\\
$1.0$  & \hgl{1.000} & 0.000 & 0.000 & 1.00 & \hgl{0.991} & 0.322 & 0.070\\
$1.5$  & \hgl{1.000} & 0.000 & 0.000 & 1.00 & \hgl{0.966} & 0.396 & 0.069\\
$2.0$  & \hgl{1.000} & 0.000 & 0.000 & 1.00 & \hgl{0.836} & 0.469 & 0.076\\
\midrule
mean   & \hgl{0.97} & 0.00 & 0.03 & 1.00 & \hgl{0.97} & 0.31 & 0.08\\
\bottomrule
\end{tabular}
\caption{$\Szero$ (planted $H_0$). The $H_0$ column is highlighted.}
\label{tab:S0}
\end{table}

\begin{table}[H]
\centering
\begin{tabular}{c ccc c ccc}
\toprule
& \multicolumn{3}{c}{selected weight $\lamstar$} & & \multicolumn{3}{c}{alignment $A(K_{H_q},Y)$}\\
\cmidrule(lr){2-4}\cmidrule(lr){6-8}
$s$ & $H_0$ & \hgl{$H_1$} & $H_2$ & hit & $H_0$ & \hgl{$H_1$} & $H_2$\\
\midrule
$0$    & 0.000 & \hgl{1.000} & 0.000 & 1.00 & 0.329 & \hgl{0.977} & 0.652\\
$0.25$ & 0.000 & \hgl{1.000} & 0.000 & 1.00 & 0.365 & \hgl{0.978} & 0.422\\
$0.5$  & 0.000 & \hgl{1.000} & 0.000 & 1.00 & 0.366 & \hgl{0.959} & 0.288\\
$0.75$ & 0.000 & \hgl{1.000} & 0.000 & 1.00 & 0.272 & \hgl{0.938} & 0.257\\
$1.0$  & 0.000 & \hgl{1.000} & 0.000 & 1.00 & 0.319 & \hgl{0.868} & 0.158\\
$1.5$  & 0.005 & \hgl{0.995} & 0.000 & 1.00 & 0.250 & \hgl{0.503} & 0.070\\
$2.0$  & 0.315 & \hgl{0.365} & 0.320 & 0.30 & 0.259 & \hgl{0.234} & 0.124\\
\midrule
mean   & 0.05 & \hgl{0.91} & 0.05 & 0.90 & 0.31 & \hgl{0.78} & 0.28\\
\bottomrule
\end{tabular}
\caption{$\Sone$ (planted $H_1$). The $H_1$ column is highlighted.}
\label{tab:S1}
\end{table}

\begin{table}[H]
\centering
\begin{tabular}{c ccc c ccc}
\toprule
& \multicolumn{3}{c}{selected weight $\lamstar$} & & \multicolumn{3}{c}{alignment $A(K_{H_q},Y)$}\\
\cmidrule(lr){2-4}\cmidrule(lr){6-8}
$s$ & $H_0$ & $H_1$ & \hgl{$H_2$} & hit & $H_0$ & $H_1$ & \hgl{$H_2$}\\
\midrule
$0$    & 0.000 & 0.000 & \hgl{1.000} & 1.00 & 0.266 & 0.317 & \hgl{0.993}\\
$0.25$ & 0.000 & 0.000 & \hgl{1.000} & 1.00 & 0.249 & 0.317 & \hgl{0.984}\\
$0.5$  & 0.000 & 0.000 & \hgl{1.000} & 1.00 & 0.235 & 0.311 & \hgl{0.982}\\
$0.75$ & 0.000 & 0.000 & \hgl{1.000} & 1.00 & 0.230 & 0.313 & \hgl{0.945}\\
$1.0$  & 0.000 & 0.000 & \hgl{1.000} & 1.00 & 0.229 & 0.302 & \hgl{0.860}\\
$1.5$  & 0.000 & 0.000 & \hgl{1.000} & 1.00 & 0.199 & 0.283 & \hgl{0.492}\\
$2.0$  & 0.065 & 0.170 & \hgl{0.765} & 0.80 & 0.215 & 0.237 & \hgl{0.263}\\
\midrule
mean   & 0.01 & 0.02 & \hgl{0.97} & 0.97 & 0.23 & 0.30 & \hgl{0.79}\\
\bottomrule
\end{tabular}
\caption{$\Stwo$ (planted $H_2$). The $H_2$ column is highlighted.}
\label{tab:S2}
\end{table}

\begin{table}[H]
\centering
\begin{minipage}[t]{0.48\textwidth}\centering
\begin{tabular}{l ccc}
\toprule
$W$ & $\to H_0$ & $\to H_1$ & $\to H_2$\\
\midrule
planted $H_0$ & \hgl{0.97} & 0.00 & 0.03\\
planted $H_1$ & 0.05 & \hgl{0.91} & 0.05\\
planted $H_2$ & 0.01 & 0.02 & \hgl{0.97}\\
\midrule
\multicolumn{4}{l}{diagonal mass $=0.95$ \;}\\
\bottomrule
\end{tabular}
\end{minipage}\hfill
\begin{minipage}[t]{0.48\textwidth}\centering
\begin{tabular}{l ccc}
\toprule
align & $\to H_0$ & $\to H_1$ & $\to H_2$\\
\midrule
planted $H_0$ & \hgl{0.97} & 0.31 & 0.08\\
planted $H_1$ & 0.31 & \hgl{0.78} & 0.28\\
planted $H_2$ & 0.23 & 0.30 & \hgl{0.79}\\
\midrule
\multicolumn{4}{l}{diagonal mass $=0.85$ \;}\\
\bottomrule
\end{tabular}
\end{minipage}
\caption{Left: weight-confusion matrix $W$ (mean selected weight). Right: centred alignment matrix for each $H_q$.}
\label{tab:confusion}
\end{table}

\section{Conclusions}
\label{sec:conclusion}

In this paper, we studied kernel $k$-means for topology-based clustering with sliced Wasserstein
kernels, and how to combine topological information across homology groups. The convex combination of SW kernels is
positive definite, inherits the stability and discriminativity of its components
(Theorems~\ref{thm:stab} and~\ref{thm:disc}), and keeps a separate bandwidth for each $q$-homology kernel. On the benchmark of~\cite{marchese2017kmeans} it improves on the
baseline, and on the synthetic circle--sphere--torus dataset~\cite{cao2025kmeans} it
is competitive with the strongest baseline. On the 3D shape-matching data~\cite{sumner2004deformation} it separates all eight classes by concentrating
weight on the $H_2$ kernel. On all three benchmarks it outperforms the kernel on the union of all homology groups and, under perturbation, the SW kernel applied
directly to point clouds. Finally, on three synthetic datasets with a topological
signal planted in a single known homology group, the weight selected by
maximising centred kernel--target alignment lands on the planted homology group with
hit rate $1.0$ up to moderate noise, giving a weight-confusion matrix with mean
diagonal entry $0.95$.

In the clustering experiments we selected the weights using the Adjusted Rand Index, which compares against the true labels, so the weight selection is supervised even though the clustering is not. Without labels, the same search can use a measure such as the silhouette score, which relies only on the 
kernel-induced distances. In the weight-significance
experiments we kept only the few most persistent points per homology group (here the
top three) as a denoising step. This works because each class here was
built to have just one or two real features, so the planted signal is always
among the longest-lived and is never discarded. But this cut-off is a
convenience tied to such datasets, not a general rule. The datasets where the discriminating
signal is instead carried by many small, short-lived features, keeping only the
longest few would throw away exactly the information that matters.

\bibliographystyle{plain}
\bibliography{references}

\end{document}